\documentclass{amsart}
\usepackage{hyperref,enumerate}
\usepackage{nicematrix}

\newcommand{\hR}{\hat{\RR}}

\def\RR{{\mathrm R}}
\def \2R{{\hat{\RR}}}
\def\WW{{\mathrm W}}

\def\BB{{\mathrm B}}

\def\EE{{\mathrm E}}

\def\TT{{\mathrm{T}}}
\def\Rc{{\mathrm {Rc}}}

\def\SS{{\mathrm S}}

\def\diag{\mathrm{Diag}}
\def\id{\mathrm{Id}}

\newtheorem{theorem}{Theorem}[section]
\newtheorem{lemma}[theorem]{Lemma}

\newtheorem{proposition}[theorem]{Proposition}

\newtheorem{conjecture}[theorem]{Conjecture}
\theoremstyle{definition}
\newtheorem{definition}[theorem]{Definition}

\theoremstyle{remark}

\numberwithin{equation}{section}

\title[]{Curvature of the Second kind \\
and a conjecture of Nishikawa }

\author{Xiaodong Cao}
\address{Department of Mathematics\\
         Cornell University \\
         Ithaca, NY  14853}

\author{Matthew J. Gursky}
\address{Department of Mathematics \\
         University of Notre Dame\\
         Notre Dame, IN 46556}

\author{Hung Tran}
\address{Department of Mathematics and Statistics\\
         Texas Tech University \\
         Lubbock, TX  79409}
         
         \date{\today}

\begin{document}

\begin{abstract} In this paper we investigate manifolds for which the curvature of the second kind (following the terminology of Nishikawa in \cite{nishi86positive}) satisfies certain positivity conditions. Our main result settles Nishikawa's conjecture that manifolds for which the curvature (operator) of the second kind are positive are  diffeomorphic to a sphere, by showing that such manifolds satisfy Brendle's \text{PIC1} condition.   In dimension four we show that curvature of the second kind has a canonical normal form, and use this to classify Einstein four-manifolds for which the curvature (operator) of the second kind is five-non-negative.  We also calculate the normal form for some explicit examples in order to show that this assumption is sharp.  
\end{abstract}

\maketitle

\section{Introduction}  \label{Intro}

Let $V$ be a an $n$-dimensional (real) inner product space, and let $\RR : \otimes^4 V  \rightarrow \mathbb{R}$ be an algebraic curvature tensor.  If $\TT^2(V)$ denotes the space of bilinear forms on $V$, then we have the splitting
\begin{align*}
\TT^2(V) = S^2(V) \oplus \Lambda^2(V),
\end{align*}
where $S^2$ is the space of symmetric two-tensors and $\Lambda^2$ is the space of two-forms.  By the symmetries of $\RR$, there are (up to sign)  two ways that $\RR$ can induce a linear map $\RR : \TT^2(V) \rightarrow \TT^2(V).$  The classical example is $R : \Lambda^2(V) \rightarrow \Lambda^2(V)$, defined by
\begin{align} \label{first}
R(e^i \wedge e^j) = \frac{1}{2} \sum_{k, \ell} R_{ijk \ell} \, e^k \wedge e^{\ell},
\end{align}
where $\{ e^1,\dots,e^n\}$ is an orthonormal basis of $V^{*}$.  When $\RR$ is the curvature tensor of a Riemannian metric, then the map (\ref{first}) is called the curvature operator.

The second map is $\hR : S^2(V) \rightarrow S^2(V)$, defined by
\begin{align} \label{second}
\hR( e^i \odot e^j ) = \sum_{k, \ell} R_{i k \ell j} \, e^k \odot e^{\ell},
\end{align}
where $\odot$ is the symmetric product (see Section \ref{Sec2} for definitions and conventions).  Note that $S^2(V)$ is not irreducible under the action of the orthogonal group on $V$.  If we let $S^2_0(V)$ denote the space of trace-free symmetric two-tensors, then $S^2(V)$ splits as
\begin{align*}
 S^2(V)=S^2_0(V)\oplus \text{Id}.
 \end{align*}
 $\hR$ induces a bilinear form $\hR : S^2_0(TM) \times S^2_0(TM) \rightarrow \mathbb{R}$ by restriction to $S^2_0(V)$.  When $\RR$ is the curvature tensor of a Riemannian metric, S. Nishikawa called $\hR$ the {\em curvature operator of the second kind}, to distinguish it from the map $\RR$ in (\ref{first}), which he called the {\em curvature operator of the first kind} (see  \cite{nishi86positive} and also \cite{bk78}).

The curvature operator of the second kind naturally arises as the term in Lichnerowicz Laplacian involving the curvature tensor, see \cite{MRS20}. As such, its sign plays a crucial role in rigidity questions for Einstein metrics.   We say that $\hat{\RR} > 0$ (respectively, $\hat{\RR} \geq 0$) if the eigenvalues of $\hat{\RR}$ as a bilinear form on $S^2_0(V)$ are positive (respectively non-negative).  It is easy to see that if $\hat{\RR} > 0$ (resp. $\geq 0$), then the sectional curvature is positive (resp., non-negative).

Nishikawa proposed the following conjecture (\cite{nishi86positive}):

\begin{conjecture}
	\label{sphere2ndkind}
	Let $(M,g)$ be a closed, simply connected Riemannian manifold. If $\hat{\RR}\geq 0$ then $M$ is diffeomorphic to a Riemannian locally symmetric space. If the inequality is strict, then $M$ is diffeomorphic to a round sphere.
\end{conjecture}

This can be viewed as a differentiable sphere conjecture for curvature of the second kind.  In dimension three, it is easy to check that $\hat{\RR}\geq 0$ implies $\Rc\geq \frac{\SS}{6}$, where $\Rc$ is the Ricci tensor and $\SS$ is the scalar curvature.  In particular the positive case of the conjecture follows from the work of Hamilton \cite{H3}.  In all dimensions, if $\hat{\RR}>0$ then $M$ is a real homology sphere \cite{OT79homology}.   Also, if one imposes additional conditions on the metric (for example, harmonic curvature), then the conjecture is true (see \cite{kashiwada93}).

Our first result is that the positive case of Nishikawa's Conjecture is true -- in fact, the assumption can be weakened:

\begin{theorem} \label{Main1} Let $(M,g)$ be a closed Riemannian manifold such that $\hat{\RR}$ is two-positive (i.e., the sum of the smallest two eigenvalues of $\hat{\RR}$ is positive).  Then $M$ is diffeomorphic to a spherical space form.
\end{theorem}

To explain the idea of the proof of Theorem \ref{Main1}, it will be helpful to recall a definition due to S. Brendle \cite{BrendleDuke}:

\begin{definition}
	\label{defPIC1}  $(M,g)$ satisfies the PIC1 condition if for any orthonormal frame $\{e_1, e_2, e_3, e_4\}$ we have
\begin{align} \label{PIC1}
\RR_{1313}+\lambda^{2}\RR_{1414}+\RR_{2323}+\lambda^{2}\RR_{2424}-2\lambda\RR_{1234} > 0 \text{ for all } \lambda\in [0,1].
\end{align}	
If the quantity in (\ref{PIC1}) is non-negative for any orthonormal frame, then we say that $(M,g)$ satisfies the NIC1 condition.
\end{definition}

PIC1 is equivalent to the condition that the product manifold $(M \times \mathbb{R}, g + ds^2)$ has positive isotropic curvature (PIC); see Proposition 4 of \cite{BrendleDuke}.  Brendle showed that if $(M,g)$ satisfies the PIC1 condition, then the Ricci flow with initial metric $g$ exists for all time and converges to a
constant curvature metric as $t \rightarrow \infty$ (see Theorem 2 of \cite{BrendleDuke}).


In earlier work of Brendle-Schoen \cite{bs091}, they proved a differentiable sphere theorem for quarter-pinched metrics.  We also remark that C. B\"ohm and B. Wilking \cite{bohmwilking} had earlier showed that if the curvature operator is two-positive, then the Ricci flow converges to a constant curvature metric.  It is not difficult to see that two-positivity of $\RR$ implies PIC1.  All of these results can be viewed as (differentiable) sphere theorems for curvature of the first kind.

To prove Theorem \ref{Main1}, we show

\begin{theorem} \label{Main2a}  Let $(M,g)$ be a Riemannian manifold of dimension $n \geq 4$ for which $\hR$ is two-positive (resp., two-non-negative).  Then $(M,g)$ satisfies PIC1 (resp. NIC1).
\end{theorem}

Theorem \ref{Main1} therefore follows from Theorem \ref{Main2a} and Theorem 2 of \cite{Brendle08PIC1}.  We will also show

\begin{theorem} \label{Main2b} Let $(M,g)$ be a Riemannian manifold of dimension $n \geq 4 $ for which $\hR$ is four-positive (respectively, four-non-negative).  Then $(M,g)$ satisfies \text{PIC} (resp., non-negative isotropic curvature).  
\end{theorem}

Combining Theorem \ref{Main2b} with the work of Micallef-Moore \cite{MM88}, we have 

\begin{theorem} \label{Main2c} Let $(M,g)$ be a simply connected Riemannian manifold of dimension $n \geq 4$ for which $\hR$ is four-positive.  Then $(M,g)$ is homeomorphic to $S^n$. 
\end{theorem}

Subsequently, Brendle showed that Einstein manifolds of dimension $n \geq 4$ with PIC have constant sectional curvature, and if $(M,g)$ has non-negative isotropic curvature, then it is locally symmetric \cite{BrendleDuke} (the four-dimensional case was earlier proved by Micallef-Wang \cite{mw93}).  Therefore, a further consequence of Theorem \ref{Main2b} is

\begin{theorem} \label{Main3}  Let $(M, g)$ be a compact Einstein manifold of dimension $n \geq 4$.  If $\hR$ is four-positive, then $(M, g)$ has constant sectional
curvature.  If $\hR$ is four-non-negative, then $(M, g)$ is locally symmetric
\end{theorem}

\subsection{Dimension Four}  For our next results we study curvature of the second kind in dimension four.  If $(M^4,g)$ is a closed, oriented four-manifold, recall that Singer-Thorpe \cite{st69} showed that the curvature operator has a canonical block decomposition of the form
	\begin{equation}\label{R4block}
	\RR = \left (\begin{array}{ll}
	\WW^{+} + \frac{1}{12}S \, I & \hskip.25in \BB \hskip.25in \\
	\hskip.25in \BB^t \hskip.25in & \WW^{-} + \frac{1}{12} S \, I,
	\end{array}
\right),
	\end{equation}
where $\WW^{\pm} : \Lambda^2_{\pm} \rightarrow \Lambda^2_{\pm}$ denotes the (anti-)self-dual Weyl tensor, and $\BB : \Lambda^2_{+} \rightarrow \Lambda^2_{-}$ is determined by the trace-free Ricci tensor, and $S$ is the scalar curvature.  In particular, $\BB$ vanishes if and only if $(M^4,g)$ is Einstein (See Section \ref{Sec2} for more details).

Analogous to this decomposition for $\RR$, we prove the following block decomposition for the matrix associated to the bi-linear form $\hat{\RR}$:

\begin{theorem} \label{Main4} Let $(M^4,g)$ be a closed, oriented four-manifold.  Then there is a basis of $S^2_0(TM^4)$ with respect to which the matrix of $\hat{\RR}$ is given by
\begin{align} \label{Rmatrix}
\hat{R} = \left (\begin{array}{lll}
  \, \mathcal{D}_1 \, & \, \mathcal{O}_1 \, & \, \mathcal{O}_2 \, \\
\, - \mathcal{O}_1 \, & \, \mathcal{D}_2 \, & \, \mathcal{O}_3  \, \\
\, - \mathcal{O}_2 \, & \, - \mathcal{O}_3 \, & \, \mathcal{D}_3 \,
\end{array}
\right),
\end{align}
and the $\mathcal{D}_i$'s are diagonal matrices given by
\begin{align} 
\label{D1matrix}
\mathcal{D}_i =  \left (\begin{array}{lll}
 \displaystyle -4( \lambda_i + \mu_1) + \frac{1}{3}S  &  &  \\
 & \displaystyle -4 ( \lambda_i  + \mu_2 ) + \frac{1}{3}S    &  \\
  &   &  \displaystyle -4 ( \lambda_i + \mu_3) + \frac{1}{3}S
\end{array}
\right),
\end{align}
where $\{ \lambda_1, \lambda_2, \lambda_3 \}$ are the eigenvalues of $\WW^{+}$, and $\{ \mu_1, \mu_2, \mu_3 \}$ are the eigenvalues of $\WW^{-}$.   Moreover, $ \mathcal{O}_1, \mathcal{O}_2, \mathcal{O}_3$ are skew-symmetric $3 \times 3$ matrices which vanish if and only if $(M^4,g)$ is Einstein.
\end{theorem}

The precise form of $ \mathcal{O}_1, \mathcal{O}_2, \mathcal{O}_3$ are given in Proposition \ref{Emat} in Section \ref{Sec4}.  If $(M^4,g)$ is Einstein then the matrix for $\hR$ is diagonal, and the eigenvalues of $\hR$ are determined by the eigenvalues of $\WW^{\pm}$ and the scalar curvature.  Using the block decomposition for $\hR$ and the work of the first and third authors, we can weaken the assumption of Theorem \ref{Main3} to show

\begin{theorem} 
\label{Main5}
	Let $(M,g)$ be a simply connected Einstein four-manifold such that $\hat{\RR}$ is five-non-negative.  Then $(M^4,g)$ is isometric, up to rescaling, to either the round sphere or complex projective space with the Fubini-Study metric.
\end{theorem}

In Section \ref{Examples} we compute the matrix explicitly for certain model cases.  For $(\mathbb{CP}^2,g_{FS})$, where $g_{FS}$ is the Fubini-Study metric, it is easy to see that $\hR$ is five-positive but not four-positive (the latter being clear from Theorem \ref{Main3}).  For $(S^2 \times S^2, g_p)$, where $g_p$ is the product metric, then $\hR$ is not five-non-negative, but is six-non-negative.  Therefore, the assumption of Theorem \ref{Main5} is sharp.

There are a number of results which classify Einstein four-manifolds under various assumptions on the curvature operator (of the first kind); see for example \cite{BrendleDuke, caotran4, cao14einstein, CR14, gl99, yangdg00} and references  therein.


\medskip

The paper is organized as follows:  in Section \ref{Sec2} we summarize the necessary background material and establish our notation and conventions.  In Section \ref{Sec3} we give the proof of Theorems \ref{Main1}, \ref{Main2a}, and \ref{Main2b}.  In Section \ref{Sec4} we give the proof of Theorem \ref{Main4}, and in Section \ref{Sec5} we prove the classification result of Theorem \ref{Main5}.    \\

{\bf Acknowledgment.} The first author acknowledges the support from the Simons Foundation (\#585201).  The second author acknowledges the support of NSF grant DMS-2105460.  The third author is partially supported by a Simons Collaboration Grant and NSF grant DMS-2104988. Also, part of the research was done when he visited the Vietnam Institute for Advanced Study in Mathematics.

\section{Preliminaries} \label{Sec2}

\subsection{Notation and conventions}  We adopt the following notation and conventions:  

\begin{itemize} 
	\item $(M^n,g)$ is a Riemannian manifold of dimension $n$. \\
	
	\item $\RR$, $\Rc$, $\SS$, and $\WW$ denote the Riemannian, Ricci, scalar, and Weyl curvatures respectively. $\EE = \Rc = \frac{1}{n} S g$ denotes the traceless Ricci tensor, and $K$ is the sectional curvature. \\

	\item Given $p \in M$, if $\{e_1,..e_n\}$ is an orthonormal basis of $T_pM$, then $\{ e^1, \dots, e^n\}$ denote the dual basis of $T^{*}_pM$. At times we may assume that these bases are locally defined via parallel transport.  \\
	
	\item The tensor product of two one-forms is defined via 
	\[(e^i\otimes e^j)(e_k, e_{\ell}) = \delta_{ik} \delta_{j \ell}.\] 
	The symmetric product of $e^i$ and $e^j$ is given by 
	\[e^i\odot e^j=e^i\otimes e^j+e^j\otimes e^i.\]
	The wedge product is given by
	\[e^i\wedge e^j=e^i\otimes e^j- e^j\otimes e^i.\]
	\ \\

	\item Let $V$ be a finite dimensional vector space. Then ${S}^2 (V)$ and $\Lambda^2(V)$ denote the space of symmetric and skew-symmetric two-tensors (i.e., bilinear forms) on $V$ (2-tensors and 2-forms, respectively). Then the space $T^2(V)$ of bilinear forms on $V$ can be decomposed as  
	\[T^2(V)={S}^2(V) \oplus \Lambda^2 (V).\]
	Also, we let $S^2_0(V)$ denote trace-free symmetric two-tensors. \\
	
	\item The inner product in $S^2(V)$ is given by 
	\begin{equation}
	\label{innerproduct1}
	\left\langle{u, v}\right\rangle = \text{Tr}(u^Tv).
	\end{equation}
	The inner product in $\Lambda^2(V)$ is given by 
	\begin{equation}
	\label{innerproduct1b}
	\left\langle{u, v}\right\rangle = \frac{1}{2}\text{Tr}(u^Tv).
	\end{equation}
	With this convention, $e_{ij}=e^i \wedge e^j$ is an orthonormal basis of $\Lambda^2$, and 
	\begin{equation}\label{innerproduct3} \alpha(e_i,e_j) =  \left\langle{\alpha,e_i\wedge e_j}\right\rangle.\end{equation}
	\ \\ 
	
	\item For $A,B \in S^2$, the {\em Kulkani-Nomizu} product $A\circ B\in S^2(\Lambda^2)$ is defined by
	\[(A\circ B)_{ijkl}=A_{ik}B_{jl}+A_{jl}B_{ik}-A_{il}B_{jk}-A_{jk}B_{il}.\]
	\ \\
	
	\item Let $\mathcal{R}(V)$ be the space of algebraic curvature tensors; i.e., $(4,0)$ tensors satisfying the the same symmetry properties as the Riemannian curvature tensor, along with the first Bianchi identity. Namely, if $T\in \mathcal{R}(V)$, then
	\begin{align*}
	T(e_i, e_j, e_k, e_l) &= -T(e_j, e_i, e_k, e_l)=-T(e_i, e_j, e_l, e_k)= T(e_k, e_l, e_i, e_j),\\
	0 &=T(e_i, e_j, e_k, e_l)+T(e_i, e_k, e_l, e_j)+T(e_i, e_l, e_j, e_k).
	\end{align*} 
	\ \\
	
	\item Any $T\in \mathcal{R}(V)$ can be identified with an element of $\mbox{End}(\Lambda^2)$:  If $\omega\in \Lambda^2$, 
	\[T(\omega)(e_i, e_j):=\sum_{k<l}T(e_i, e_j, e_k, e_l)\omega (e_k, e_l).\]
	As a consequence,
	\begin{align}
	T_{ijkl}&:=T(e_i, e_j, e_k, e_l)\nonumber\\
	\label{innerproduct4}
	&= T(e^i\wedge e^j, e^k\wedge e^l):=\left\langle{T(e^i\wedge e^j), e^k\wedge e^l}\right\rangle.\end{align} 
	\ \\
	
	\item Any $T\in \mathcal{R}(V)$ can also be identified with an element of $\mbox{End}(S^2)$:  If $A\in S^2$, 
\[(\hat{T}A)(e_i, e_k)=\sum_{j,l} T(e_i, e_j, e_l, e_k)A(e_j, e_l).\]
To distinguish this identification from the previous one, we denote the latter by $\hat{T}$ and refer to it as the {\em curvature operator of the second kind.}   

Of course the case of interest to us is When $T = \RR$, the Riemannian curvature tensor of $(M,g)$.  We say that the (Riemannian) curvature operator of the second kind $\hat{\RR}$ is $k$-positive (non-negative) if the sum of any $k$ eigenvalues of $\hat{\RR} \vert_{S^2_0}$ is positive (non-negative).

\end{itemize}

\medskip 

\subsection{Curvature Decomposition} Recall the Riemannian curvature tensor can be decomposed into the Weyl, the Ricci, and the scalar parts.  In terms of the Kulkarni-Nomizu product defined above, we can express this decomposition as  
\begin{equation}
\label{Rdecom}
\RR=\WW+\frac{1}{n-2}\EE\circ g+\frac{\SS}{24}g\circ g.\end{equation}

In dimension four this decomposition gives rise to a decomposition of the curvature operator; see \cite{st69}.  If $(M^4,g)$ is oriented, then the Hodge star operator $* : \Lambda^2 \rightarrow \Lambda^2$, where $\Lambda^2$ is the bundle of two-forms, and induces a splitting
\begin{equation*}
\Lambda^2 =\Lambda^2_{+} \oplus \Lambda^2_{-},
\end{equation*}
where $\Lambda^2_{\pm}$ are the $\pm 1$-eigenspaces of $*$.  With respect to this splitting, the components of the splitting in (\ref{Rdecom}) have the property that 
\begin{align*}
\WW : \Lambda^2_{\pm} &\rightarrow \Lambda^2_{\pm}, \\
\EE \circ g : \Lambda^2_{\pm} &\rightarrow \Lambda^2_{\mp}.
\end{align*}
Consequently, the curvature operator $R : \Lambda^2 \rightarrow \Lambda^2$ has the following block decomposition:
\begin{equation}
\label{Hodgecuvdecom}
\RR =\left( \begin{array}{cc}
\frac{\SS}{12}\id+\WW^+ & \frac{1}{2}\EE\circ g \\
\frac{1}{2}\EE\circ g & \frac{\SS}{12}\id+\WW^-
\end{array} \right),
\end{equation}
where $\WW^{\pm}$ are the restriction of $\WW$ to $\Lambda^2_{\pm}M$.   

We will also need a related normal form for $\RR$ due to M. Berger \cite{berger61}: 

\begin{proposition}
	\label{berger}
	Let $(M,\ g)$ be a four-manifold. At each point $p\in M$, there exists an orthonormal basis $\{e_i\}_{1\leq i\leq
		4}$ of $T_p M$, such that relative to the corresponding basis
	$\{e_i\wedge e_j\}_{1\leq i<j\leq 4}$ of $\wedge^2 T_pM$,
	$\WW$ takes the form
	\begin{equation}
	\label{abba}
	\WW=\left( \begin{array}{cc}
	A & B\\
	B & A
	\end{array}\right),
	\end{equation}
	where $A=\diag\{a_1,\ a_2,\ a_3\}$, $B=\diag \{b_1,\ b_2,\ b_3\}$.
	Moreover, we have the followings:
	
	\begin{enumerate}
		\item  $a_1=\WW(e_1, e_2, e_1, e_2)=\WW(e_3, e_4, e_3, e_4)=\min_{|a|=|b|=1,~ a\perp b}\WW(a,b,a,b)$,
		\item $a_3=\WW(e_1, e_4, e_1, e_4)=\WW(e_1, e_4, e_1, e_4)=\max_{|a|=|b|=1,~ a\perp b}\WW(a,b,a,b)$.
		
		\item $a_2=\WW(e_1, e_3, e_1, e_3)=\WW(e_2, e_4, e_2, e_4)$, 
		
		\item  $b_1=\WW_{1234},\ b_2=\WW_{1342},\ b_3=\WW_{1423}$,
		\item $a_1+a_2+a_3=b_1+b_2+b_3=0$,
		
		\item  $|b_2-b_1|\leq a_2-a_1,\ |b_3-b_1|\leq a_3-a_1,\ |b_3-b_2|\leq
		a_3-a_2$.
	\end{enumerate}
\end{proposition}

\medskip 

\subsection{Isotropic Curvature}
Next we recall the notion of isotropic curvature and related concepts. The notion of isotropic curvature on 2-planes was introduced by M. Micallef and J. D. Moore in \cite{MM88}.  As mentioned in the Introduction, it played a crucial role in the proof of the differentiable sphere conjecture \cite{bs091} via the Ricci flow. 

\begin{definition}
	\label{PIC}
	$(M,g)$ is said to have non-negative isotropic curvature if, for any orthonormal frame $\{e_1, e_2, e_3, e_4\}$ we have
	\[\RR_{1313}+\RR_{1414}+\RR_{2323}+\RR_{2424}-2\RR_{1234}\geq 0.\]
	If the inequality is strict then it is said to have positive isotropic curvature. 
\end{definition}

The following property is well known (see \cite{MM88}):  

\begin{lemma}
	\label{isolemma1}
	In dimension four, non-negative isotropic curvature is equivalent to
	\[ -\WW^{\pm} + \frac{\SS}{12}\id \geq 0,\]
	as a bilinear form on $\Lambda^2_{\pm}$. 
\end{lemma}

In the work of Brendle and Schoen, they introduced the following extensions of the notion of non-negative and positive isotropic curvature:   
\begin{definition} 
	\label{PIC1}
	$(M,g)$ is said to be NIC1 if for any orthonormal frame $\{e_1, e_2, e_3, e_4\}$ we have
\[\RR_{1313}+\lambda^{2}\RR_{1414}+\RR_{2323}+\lambda^{2}\RR_{2424}-2\lambda\RR_{1234}\geq 0 \text{ for all } \lambda\in [0,1].\]
If the inequality is strict then $(M,g)$ is said to be PIC1. 	
\end{definition}

\begin{definition} 
	\label{PIC2}
	$(M,g)$ is said to be NIC2 if for any orthonormal frame $\{e_1, e_2, e_3, e_4\}$ we have
	\[\RR_{1313}+\lambda^{2}\RR_{1414}+\mu^{2}\RR_{2323}+\lambda^{2}\mu^{2}\RR_{2424}-2\lambda\mu\RR_{1234}\geq 0 \text{ for all }\lambda, \mu\in [0,1].\]
	If the inequality is strict then $(M,g)$ is said to be PIC2. 	
\end{definition}

Brendle and Schoen observed that all these conditions are preserved under the Ricci flow \cite{bs091, Brendle08PIC1, Brendle18PIC}. In particular, they were able to show the following:  
\begin{theorem}[\cite{Brendle08PIC1}]
	\label{classificationPIC1}
	Let $(M, g)$ be a Riemannian manifold satisfying the PIC1 condition. Then the normalized Ricci flow exists for all time and converges to a constant curvature metric as $t\rightarrow \infty$. In particular, the manifold is diffeomorphic to a spherical space form.  
\end{theorem}

\section{Curvature of the second kind and PIC}  \label{Sec3}

In this section, we give the proofs to Theorems \ref{Main1}, \ref{Main2a}, and \ref{Main2b}.  

\begin{proof}[Proof of Theorem \ref{Main2a}] Fix a point $p \in M$ and let $\{e^1,...e^n\}$ be an orthonormal basis of $T_p^{*}M$.  We define the following trace-free symmetric two tensors:
	\begin{align*}
		h_1 &= e^1\odot e^3 + \lambda e^2\odot e^4,\\
		h_2 &= e^2\odot e^3 - \lambda e^1\odot e^4.
	\end{align*}
	It is easy to see that $h_1$ and $h_2$ are orthogonal to each other in $S^2$.
	Since $\hat{\RR}$ is two-positive we have
	\begin{align*}
		0 &< \hat{\RR}(h_1, h_1)+  \hat{\RR}(h_2, h_2).
	\end{align*}

	We observe that all components of $h_1$ are trivial except
	\begin{align*}
		h_1(e^1, e^3):= (h_1)_{13} &=(h_1)_{31}=1,\\
		h_1(e^2, e^4):= (h_1)_{24} &=(h_2)_{42}=\lambda.
	\end{align*}
	Then, we calculate
	\begin{align*}
		\2R(h_1, h_1) &= \sum_{ijkl} \RR_{ijkl} (h_1)_{il}(h_1)_{jk} \\
		&= \sum_{i, j, k, l, |l-i|=|k-j|=2} \RR_{ijkl}(h_1)_{il}(h_1)_{jk}\\
		&= 2(2\lambda\RR_{1243}+\RR_{1313}+2\lambda\RR_{1423}+\lambda^2\RR_{2424}).
	\end{align*}
	
	Similarly,
	\begin{align*}
		(h_2)_{23} &= (h_2)_{32}=1,\\
		(h_2)_{14} &= (h_2)_{41}=-\lambda.
	\end{align*}
	Then, we calculate
	\begin{align*}
		\2R(h_2, h_2) &=\sum_{ijkl} \RR_{ijkl} (h_2)_{il}(h_2)_{jk} \\
		&= \sum_{i, j, k, l, l+i=k+j=5} \RR_{ijkl}(h_2)_{il}(h_2)_{jk}\\
		&= 2(-2\lambda\RR_{1234}-2\lambda\RR_{1324}+\lambda^2\RR_{1414}+\RR_{2323}).
	\end{align*}
	
	Combining equations above yields
	\begin{align*}
	0 &< (2\lambda\RR_{1243}+\RR_{1313}+2\lambda\RR_{1423}+\lambda^2\RR_{2424})\\
	&+ (-2\lambda\RR_{1234}-2\lambda\RR_{1324}+\lambda^2\RR_{1414}+\RR_{2323})\\
	&= \RR_{1313}+\RR_{2323}+\lambda^2 (\RR_{1414}+\RR_{2424})-4\lambda\RR_{1234}-2\lambda(\RR_{1432}+\RR_{1324}).
	\end{align*}
	Applying the first Bianchi identity, we obtain
	\begin{equation}
	\label{2poseq1}
	0< (\RR_{1313}+\RR_{2323})+\lambda^2(\RR_{1414}+\RR_{2424})-6\lambda\RR_{1234}.
	\end{equation}
	Interchanging the roles of $e^1$ and $e^2$ and letting 
	\begin{align*}
		h_3 &= e^2\odot e^3+\lambda e^1\odot e^4,
		\end{align*}
		we have 
		\begin{align*}
		\2R(h_3,h_3) &=2(\RR_{2323}+\lambda^2\RR_{1414}+2\lambda\RR_{1324}+2\lambda\RR_{2143}).
		\end{align*}
		Similarly,
	\begin{align*}
		h_4 &= e^1\odot e^3-\lambda e^2\odot e^4,\\
		\2R(h_4,h_4) &=2(\RR_{1313}+\lambda^2\RR_{2424}-2\lambda\RR_{1423}-2\lambda\RR_{2134})
	\end{align*}
	Adding these results together, we obtain
	\begin{align*}
		0 &<(2\lambda\RR_{2143}+\RR_{2323}+2\lambda\RR_{1324}+\lambda^2\RR_{1414})\\
		&+ (-2\lambda\RR_{2134}-2\lambda\RR_{1423}+\lambda^2\RR_{2424}+\RR_{1313})\\
		&= \RR_{1313}+\RR_{2323}+\lambda^2 (\RR_{1414}+\RR_{2424})-4\lambda\RR_{2134}-2\lambda(\RR_{1423}+\RR_{1342}).		
	\end{align*}
Applying the first Bianchi identity, we obtain
\begin{equation}
\label{2poseq2}
0< (\RR_{1313}+\RR_{2323})+\lambda^2(\RR_{1414}+\RR_{2424})-6\lambda\RR_{2134}.
\end{equation}
From equations (\ref{2poseq1}) and (\ref{2poseq2}), one concludes that
	\[\RR_{1313}+\RR_{2323}+\lambda^2\RR_{1414}+\lambda^2\RR_{2424}> |6\lambda \RR_{1234}|.\]
		By Defintion \ref{PIC1}, the \text{PIC1} condition is equivalent to
	\[ \RR_{1313}+\RR_{2323}+\lambda^2\RR_{1414}+\lambda^2\RR_{2424}+2\lambda\RR_{1234} > 0.\]
	The result then follows.
\end{proof}

\medskip 

\begin{proof}[Proof of Theorem \ref{Main1}] 
	By Theorem \ref{Main2a}, the curvature is PIC1. The result follows from Theorem \ref{classificationPIC1}.
\end{proof}

\medskip

\begin{proof}[Proof of Theorem \ref{Main2b}] 
 
	As before, we fix a point $p \in M$ and let $\{e^1,...e^n\}$ be an orthonormal of $T_p^{*}M$. We define the following traceless symmetric two tensors:
	\begin{align*}
	h_1 &= \frac{1}{2}(-e^1\odot e^1- e^2\odot e^2+e^3\odot e^3+e^4\odot e^4),\\
	h_2 &= e^1\odot e^4- e^2\odot e^3,\\
	h_3 &= -e^1\odot e^3- e^2\odot e^4,\\
	h_4 &= -e^1\odot e^4- e^2\odot e^3,\\
	h_5 &= \frac{1}{2}(-e^1\odot e^1+e^2\odot e^2-e^3\odot e^3+e^4\odot e^4),\\
	h_6 &= \frac{1}{2}(-e^1\odot e^1+ e^2\odot e^2+e^3\odot e^3-e^4\odot e^4).		
	\end{align*}
	It is easy to see that these tensors are of the same magnitude and are  mutually orthogonal in $S^2$.
	
	Since $\hat{\RR}$ is 4-positive we have
	\begin{align*}
		0 &< \hat{\RR}(h_1, h_1)+  \hat{\RR}(h_2, h_2)+ \hat{\RR}(h_4, h_4)+ \hat{\RR}(h_5, h_5).
	\end{align*}
	We compute
	\begin{align*}
	 \hat{\RR}(h_1, h_1) &= \sum_{ijkl} \RR_{ijkl} (h_1)_{il}(h_1)_{jk} \\
	 &= \sum_{i, j}\RR_{ijji}(h_1)_{ii}(h_1)_{jj}\\
	 &= 2(-\RR_{1212}-\RR_{3434}+\RR_{1313}+\RR_{1414}+\RR_{2323}+\RR_{2424}).
	\end{align*}
	Next,
	\begin{align*}
		 \hat{\RR}(h_2, h_2) &= \sum_{ijkl} \RR_{ijkl} (h_2)_{il}(h_2)_{jk} \\
		 &= \sum_{i, j}\RR_{ij(5-j)(5-i)}(h_2)_{i(5-i)}(h_2)_{j(5-j)}\\
		 &= 2(\RR_{1414}+\RR_{2323}+2\RR_{1243}+2\RR_{1342}).
	\end{align*}
	Similarly,
	\begin{align*}
		\hat{\RR}(h_4, h_4) &= 2(-2\RR_{1243}-2\RR_{1342}+\RR_{1414}+\RR_{2323}),\\
		\hat{\RR}(h_5, h_5) &= 2(-\RR_{1313}-\RR_{2424}+\RR_{1414}+\RR_{2323}+\RR_{1212}+\RR_{3434}).\\
	\end{align*}
	Combining equations above yields
	\begin{equation}
	\label{4poseq1}
	0 < \RR_{1414}+\RR_{2323}.
	\end{equation}
	Next, we consider
	\[0 <  \hat{\RR}(h_1, h_1)+  \hat{\RR}(h_2, h_2)+ \hat{\RR}(h_3, h_3)+ \hat{\RR}(h_6, h_6).\]
	Here,
	\begin{align*}
	\hat{\RR}(h_3, h_3) &= 2(-2\RR_{1234}-2\RR_{1432}+\RR_{1313}+\RR_{2424}),\\
	\hat{\RR}(h_6, h_6) &=  2(-\RR_{1414}-\RR_{2323}+\RR_{1313}+\RR_{1212}+\RR_{3434}+\RR_{2424}).	
	\end{align*}
	Therefore, combining equations above yield
	\begin{align*}
	0 &< (\RR_{1313}+\RR_{1414}+\RR_{2323}+\RR_{2424}-\RR_{1212}-\RR_{3434})\\
	&+(2\RR_{1243}+2\RR_{1342}+\RR_{1414}+\RR_{2323})\\
	&+(\RR_{1313}+\RR_{2424}-2\RR_{1234}-2\RR_{1432})\\
	&+(\RR_{1313}+\RR_{1212}+\RR_{2424}+\RR_{3434}-\RR_{1414}-\RR_{2424})\\
	&= 3(\RR_{1313}+\RR_{2424})+(\RR_{1414}+\RR_{2323})-4\RR_{1234}-2(\RR_{1324}+\RR_{432}).
	\end{align*}
	Applying the first Bianchi identity, we obtain
	\begin{equation}
	\label{4poseq2}
	0< 3(\RR_{1313}+\RR_{2424})+(\RR_{1414}+\RR_{2323})-6\RR_{1234}.
	\end{equation}
	Adding (\ref{4poseq2}) and twice of (\ref{4poseq1}) gives
	\[0 < 3(\RR_{1313}+\RR_{1414}+\RR_{2323}+\RR_{2424}-2\RR_{1234}).\]
	Since the inequality holds for any orthonormal four-tuple $(e_1, e_2, e_3, e_4)$, we conclude that the manifold has positive isotropic curvature.

\end{proof}

As explained in the Introduction, Theorems \ref{Main2c} and \ref{Main3} follow from Theorem \ref{Main2b} and Micallef-Wang's work \cite{MM88} and Brendle's classification of Einstein manifold with non-negative isotropic curvature \cite{BrendleDuke}.

\section{Dimension four: matrix representation of $\hR$} \label{Sec4}

Let $(M^4,g)$ be an oriented Riemannian four-manifold, and $p \in M^4$.  The space of two forms $\Lambda^2(T_p M^4)$ splits into the space of self-dual and anti-self-dual two-forms:
\begin{align*}
\Lambda^2(T_p M^4) = \Lambda_{+}^2(T_p M^4) \oplus \Lambda_{-}^2(T_p M^4)
\end{align*}
If $\{ e^1 , e^2, e^3, e^4 \}$ is an orthonormal basis of $T_p^{*}X^4$, then the two-forms
\begin{align} \label{ombase} \begin{split}
\omega^1 &= (e^1 \wedge e^2 + e^3 \wedge e^4), \\
\omega^2 &= (e^1 \wedge e^3 - e^2 \wedge e^4), \\
\omega^3 &= (e^1 \wedge e^4 + e^2 \wedge e^3),
\end{split}
\end{align}
constitute an orthogonal basis of $\Lambda_{+}^2(T_p M^4)$ with $|\omega^{\alpha}|^2 = 2$, and
\begin{align} \label{etbase} \begin{split}
\eta^1 &= (e^1 \wedge e^2 - e^3 \wedge e^4), \\
\eta^2 &= (e^1 \wedge e^3 + e^2 \wedge e^4), \\
\eta^3 &= (e^1 \wedge e^4 - e^2 \wedge e^3),
\end{split}
\end{align}
is an orthogonal basis of $\Lambda_{-}^2(T_p M^4)$ with $|\eta^{\beta}|^2 = 2$.

The Weyl tensor of $(M^4,g)$ defines trace-free (symmetric) linear endomorphisms $W^{\pm} : \Lambda_{\pm}^2(T_p M^4) \rightarrow \Lambda_{\pm}^2(T_p M^4)$, hence there are bases of $\Lambda_{\pm}^2(T_p M^4)$ consisting of eigenforms of $W^{\pm}$. Indeed, using Proposition \ref{berger}, we have
\begin{equation}
\WW=
\left( \begin{array}{cc}
(A+B) & 0 \\
0 & (A-B) \end{array} \right).\end{equation}
Here, $A=\text{diag}(a_{1}, a_{2}, a_{3})$, $B=\text{diag}(b_{1}, b_{2}, b_{3})$, and $a_1+a_2+a_3=b_1+b_2+b_3=0$.

As a result, eigenvalues of $\WW^{\pm}$ are ordered,
\begin{equation}
\label{eigenvalue}
\begin{cases}
\lambda_1=a_1+b_1 \leq \lambda_2=a_2+b_2 \leq \lambda_3=a_3+b_3, &\\
\mu_1=a_1-b_1 \leq \mu_2=a_2-b_2 \leq \mu_3=a_3-b_3. &
\end{cases}
\end{equation}

The following result is an excerpt from \cite{Derd83}, and is based on \cite{st69}:

\begin{proposition} \label{DerProp} Let $(M^4,g)$ be an oriented, four-dimensional Riemannian manifold, and $p \in M^4$.  \smallskip

\noindent $(i)$  There is an orthogonal basis of $\Lambda_{+}^2(T_p M^4)$ (respectively, $\Lambda_{-}^2(T_p M^4)$) consisting of eigenforms $\{ \omega^1 , \omega^2, \omega^3 \}$ (resp., $\{ \eta^1, \eta^2, \eta^3 \}$) of $W^{+}$ (resp., $W^{-}$) of the form (\ref{ombase}) (resp., of the form (\ref{etbase})) for some choice of basis  $\{ e^1, \dots, e^4 \}$ of $T_p^{*} M^4$.   \smallskip

\noindent $(ii)$  If $\{ \lambda_1, \lambda_2, \lambda_3\}$ and $\{ \mu_1, \mu_2, \mu_3 \}$ are the eigenvalues of $W^{+}$ and $W^{-}$ respectively, then with respect to these
bases the Weyl tensor is given by
\begin{align} \label{Wform}  \begin{split}
W_{ijk\ell} &= \frac{1}{2} \big[ \lambda_1 \omega^1_{ij} \omega^1_{k \ell} + \lambda_2 \omega^2_{ij} \omega^2_{k \ell} + \lambda_3 \omega^3_{ij} \omega^3_{k \ell} \big] + \frac{1}{2} \big[ \mu_1 \eta^1_{ij} \eta^1_{ k \ell} + \mu_2 \eta^2_{ij} \eta^2_{k \ell} + \mu_3 \eta^3_{ij} \eta^3_{k \ell} \big],
\end{split}
\end{align}
with
\begin{align} \label{tfW1} \begin{split}
\lambda_1 + \lambda_2 + \lambda_3 &= 0, \\
\mu_1 + \mu_2 + \mu_3 &= 0.
\end{split}
\end{align}
\smallskip

\noindent $(iii)$  The bases in (\ref{ombase}) and (\ref{etbase}) have a quaternionic structure:  For $1 \leq \alpha \leq 3$,
\begin{align} \label{omsq} \begin{split}
[(\omega^{\alpha})^2]_{ij} = \omega^{\alpha}_{ik} \omega^{\alpha}_{kj} = -\delta_{ij}, \\
[(\eta^{\alpha})^2]_{ij} = \eta^{\alpha}_{ik} \eta^{\alpha}_{kj} = -\delta_{ij},
\end{split}
\end{align}
where the components are with respect to an orthonormal basis of $T_p M^4$.  Also,
\begin{align} \label{omquat} \begin{split}
(\omega^1 \omega^2)_{ij} &= \omega^1_{ik} \omega^2_{k j} = - \omega^3_{ij}, \\
(\omega^1 \omega^3)_{ij} &= \omega^1_{ik} \omega^3_{k j} =  \omega^2_{ij}, \\
(\omega^2 \omega^3)_{ij} &= \omega^2_{ik} \omega^3_{k j} = - \omega^1_{ij}, \\
(\eta^1 \eta^2)_{ij} &= \eta^1_{ik} \eta^2_{k j} = \eta^3_{ij}, \\
(\eta^1 \eta^3)_{ij} &= \eta^1_{ik} \eta^3_{k j} = - \eta^2_{ij}, \\
(\eta^2 \eta^3)_{ij} &= \eta^2_{ik} \eta^3_{k j} =  \eta^1_{ij}.
\end{split}
\end{align}
\smallskip

\noindent $(iv)$ The bases in (\ref{ombase}) and (\ref{etbase}) generate an orthogonal basis of $S^2_0(T_p^{*}X^4)$, the space of symmetric trace-free
$(0,2)$-tensors by taking
\begin{align} \label{hdef}
h^{(\alpha, \beta)}_{ij}  = \omega^{\alpha}_{ik} \eta^{\beta}_{kj}.
\end{align}
Moreover, $| h^{(\alpha,\beta)}| = 2$.
\end{proposition}

To simplify notation we label the basis in Proposition \ref{DerProp} $(iv)$ in the following way:
\begin{align} \label{htoh} \begin{split}
h^{(1,1)} &= h^1, \ h^{(1,2)} = h^2, \ h^{(1,3)} = h^3, \\
h^{(2,1)} &= h^4, \ h^{(2,2)} = h^5, \ h^{(2,3)} = h^6, \\
h^{(3,1)} &= h^7, \ h^{(3,2)} = h^8, \ h^{(3,3)} = h^9.
\end{split}
\end{align}
Using the quaternionic structure of the bases of eigenforms, it is easy (but tedious) to construct a `multiplication table' for the basis element $\{ h^{\alpha} \}_{\alpha = 1}^9$:

\begin{lemma} \label{multLemma}
The basis elements in (\ref{htoh}) satisfy
\begin{center}
\begin{NiceTabular}{ |c|c|c|c|c|c|c|c|c|c|} 
 \hline
 & $h^1$ & $h^2$ & $h^3$ & $h^4$ & $h^5$ & $h^6$ & $h^7$ & $h^8$ & $h^9$ \\ \hline 
 $h^1$ & $\text{Id}$ & $\ast$ & $\ast$ & $\ast$ & $-h^9$ & $h^8$ & $\ast$ & $h^6$ & $-h^5$\\ \hline
 $h^2$ & $\ast$ & $\text{Id}$ & $\ast$ & $h^9$ & $\ast$ & $-h^7$ & $-h^6$ & $\ast$ & $h^4$\\ \hline
 $h^3$ & $\ast$ & $\ast$ & $\text{Id}$ & $-h^8$ & $h^7$ & $\ast$ & $h^5$ & $-h^4$ & $\ast$\\ \hline
 $h^4$ & $\ast$ & $h^9$ & $-h^8$ & $\text{Id}$ & $\ast$ & $\ast$ & $\ast$ & $-h^3$ & $h^2$\\ \hline
 $h^5$ & $-h^9$ & $\ast$ & $h^7$ & $\ast$ & $\text{Id}$ & $\ast$ & $h^3$ & $\ast$ & $-h^1$\\ \hline
 $h^6$ & $h^8$ & $-h^7$ & $\ast$ & $\ast$ & $\ast$ & $\text{Id}$ & $-h^2$ & $h^1$ & $\ast$\\ \hline
 $h^7$ & $\ast$ & $-h^6$ & $h^5$ & $\ast$ & $h^3$ & $-h^2$ & $\text{Id}$ & $\ast$ & $\ast$ \\ \hline
 $h^8$ & $h^6$ & $\ast$ & $-h^4$ & $-h^3$ & $\ast$ & $h^1$ & $\ast$ & $\text{Id}$ & $\ast$\\ \hline
 $h^9$ & $-h^5$ & $h^4$ & $\ast$ & $h^2$ & $-h^1$ & $\ast$ & $\ast$ & $\ast$ & $\text{Id}$\\
 \hline
\end{NiceTabular}
\end{center}
That is, 
\begin{align*}
(h^{\alpha})_{ij}^2 = h^{\alpha}_{ik} h^{\alpha}_{kj} = \delta_{ij}, 
\end{align*}
and:
\begin{align*}
h^1 h^5 &= - h^9, \ \ \ h^1 h^6 = h^8, \\
h^1 h^8 &= h^6, \ \ \ h^1 h^9 = -h^5, \\
h^2 h^4 &= h^9, \ \ \ h^2 h^6 = - h^7, \\
h^2 h^7 &= - h^6, \ \ \ h^2 h^9 = h^4 \\
h^3 h^4 &= - h^8 \ \ \ h^3 h^5 = h^7, \\
h^3 h^7 &= h^5, \ \ \ h^3 h^8 = - h^4 \\
h^4 h^8 &= -h^3 \ \ \ h^4 h^9 = h^2 \\
h^5 h^7 &= h^3 \ \ \ h^5 h^9 = -h^1 \\
h^6 h^7 &= - h^2 \ \ \ h^6 h^8 = h^1,
\end{align*}
Also, each $\ast$ represents a skew-symmetric matrix. 

\end{lemma}

As explained in the Introduction, the Weyl tensor can also be interpreted as a symmetric bilinear linear form on the space of trace-free symmetric two-tensors.  If $s,t \in S^2_0(T^{*}X^4)$, then
\begin{align} \label{Ws2}
\hat{W}(s,t) = W_{i k \ell j} s_{k \ell} t_{ij},
\end{align}
where the components are with respect to an orthonormal basis of $T_p M^4$.  We can compute the matrix of $\hat{W}$ with respect to the basis $\{ h^{\alpha} \}_{\alpha = 1}^9$, by using the algebraic properties summarized in Proposition \ref{DerProp} and Lemma \ref{multLemma}:

\begin{proposition} \label{Wmat} The basis in (\ref{hdef}) diagonalizes the Weyl tensor, interpreted as a symmetric bi-linear form as in (\ref{Ws2}). With respect to this basis the matrix of $W$ is given by
\begin{align} \label{Wmatrix}
\hat{W} = \left (\begin{array}{lll}
  \, \mathcal{D}_1 \, & \, 0 \, & \, 0 \, \\
\, 0 \, & \, \mathcal{D}_2 \, & \, 0  \, \\
\, 0 \, & \, 0 \, & \, \mathcal{D}_3 \,
\end{array}
\right),
\end{align}
where the $\mathcal{D}_i$'s are diagonal matrices given by
\begin{align} \label{D1matrix}
\mathcal{D}_i =  \left (\begin{array}{lll}
 \displaystyle -4( \lambda_i + \mu_1)  &  &  \\
 & \displaystyle -4 ( \lambda_i + \mu_2 )    &  \\
  &   &  \displaystyle -4 ( \lambda_i + \mu_3)
\end{array}
\right).
\end{align}
\end{proposition}

To express the matrix for $\hR$, we use the decomposition of the curvature tensor in four dimensions:
\begin{align} \label{rot1}
R_{i k \ell j} = W_{i k \ell j} + \frac{1}{2} \left( g_{i \ell} E_{k j} - g_{ij} E_{k \ell} - g_{k e\ll} E_{ij} + g_{kj} E_{i \ell} \right) + \frac{1}{12} S \left( g_{i \ell} g_{k j } - g_{ij} E_{k \ell} \right).
\end{align}
If $s$ and $t$ are trace-free symmetric two-tensors, then
\begin{align} \label{RHH} \begin{split}
\hR(s,t) &= R_{i k \ell j} s_{k \ell} t_{ij} \\
&= \hat{W}(s,t) + \hat{E}(s,t) + \frac{1}{12} S \langle s, t \rangle,
\end{split}
\end{align}
where $\langle \cdot , \cdot \rangle$ is the inner product on symmetric two-tensors, and $\hat{E}$ is the bilinear form given by
\begin{align} \label{Edef}
\hat{E}(s,t) = E_{ij} s_{ik} t_{k j} = \langle E , s \, t \rangle,
\end{align}
where $(s \, t)_{ij} = s_{ik} t_{kj}$.  Consequently, to compute the matrix for $\hR$ it only remains to compute the matrix for $\hat{E}$ with respect to the basis $\{ h^{\alpha} \}$.

Since $\{ h^{\alpha} \}$ is a basis for the space of trace-free symmetric two-tensors, we can write
\begin{align} \label{Eexp}
E_{ij} = \frac{1}{4} \epsilon_{\gamma} h^{\gamma}_{ij},
\end{align}
where
\begin{align} \label{epdef}
\epsilon_{\alpha} = \langle E, h^{\alpha} \rangle.
\end{align}
It follows from (\ref{Edef}) that the matrix entry $\hat{E}_{\alpha \beta} = \hat{E}(h^{\alpha}, h^{\beta})$ is given by
\begin{align} \label{Eab} \begin{split}
\hat{E}_{\alpha \beta} &= E_{ij} h_{ik}^{\alpha} h_{kj}^{\beta} \\
&= \frac{1}{4} \epsilon_{\gamma} h^{\gamma}_{ij} h^{\alpha}_{ik} h^{\beta}_{kj} \\
&= \frac{1}{4} \epsilon_{\gamma} \langle h^{\gamma}, h^{\alpha} h^{\beta} \rangle.
\end{split}
\end{align}
Using the product formulas in Lemma \ref{multLemma}, we can therefore express the entries of the matrix $(\hat{E}_{\alpha \beta})$ in terms of the $\epsilon_{\gamma}$'s:

\begin{proposition}  \label{Emat} With respect to the basis in (\ref{hdef}), the matrix of $\hat{E}$ is given by
\begin{align} \label{Ematrix}
\hat{E} = \left (\begin{array}{lll}
  \, 0 \, & \, \mathcal{O}_1 \, & \, \mathcal{O}_2 \, \\
\, - \mathcal{O}_1 \, & \, 0 \, & \, \mathcal{O}_3  \, \\
\, - \mathcal{O}_2 \, & \, - \mathcal{O}_3 \, & \, 0 \,
\end{array}
\right),
\end{align}
where $ \mathcal{O}_1, \mathcal{O}_2, \mathcal{O}_3$ are skew-symmetric $3 \times 3$ matrices given by
\begin{align} \label{Amatrix1}
\mathcal{O}_1 =  \left (\begin{array}{lll}
 0 & -\epsilon_9 & \epsilon_8 \\
\epsilon_9 &  0  & -\epsilon_7  \\
-\epsilon_8 & \epsilon_7 & 0
\end{array}
\right),
\end{align}
\begin{align} \label{Bmatrix1}
\mathcal{O}_2 = \left (\begin{array}{lll}
 0 & \epsilon_6 & -\epsilon_5 \\
-\epsilon_6 &  0  & \epsilon_4  \\
\epsilon_5 & -\epsilon_4 & 0
\end{array}
\right),
\end{align}
\begin{align} \label{Cmatrix1}
\mathcal{O}_3 =  \left (\begin{array}{lll}
 0 & -\epsilon_3 & \epsilon_2 \\
\epsilon_3 &  0  & -\epsilon_1  \\
-\epsilon_2 & \epsilon_1 & 0
\end{array}
\right).
\end{align}
Moreover, these matrices all vanish if and only if $(M^4,g)$ is Einstein.
\end{proposition}

\begin{proof}  This is a straightforward calculation, so we only point out some readily observed features.  First, since $(h^{\alpha})^2 = I$, all diagonal entries vanish:
\begin{align*}
\hat{E}(h^{\alpha}, h^{\alpha}) = \langle E , (h^{\alpha})^2 \rangle = \langle E , I \rangle = \mbox{tr }E = 0.
\end{align*}
In fact, if $1 \leq \alpha, \beta \leq 3$ and $\alpha \neq \beta$, then by Lemma \ref{multLemma} the product $h^{\alpha} \, h^{\beta}$ is skew-symmetric, hence
\begin{align*}
\hat{E}(h^{\alpha}, h^{\beta}) = \langle E , h^{\alpha} h^{\beta} \rangle = 0,
\end{align*}
since $E$ is symmetric.  This shows that the upper left $3 \times 3$ block of the matrix vanishes, and a similar argument shows that all three such blocks along the diagonal are zero.

Finally, note that all matrices vanish if and only if $\epsilon_1 = \cdots = \epsilon_9$, which by (\ref{Eexp}) is equivalent to $E = 0$.
\end{proof}

\begin{proof}[Proof of Theorem \ref{Main4}]  Theorem \ref{Main3} follows from Proposition \ref{Wmat}, Proposition \ref{Emat}, and the formula (\ref{RHH}).
\end{proof}

\section{Einstein Four Manifolds}  \label{Sec5}
In this section we apply our matrix representation of the curvature of second kind to study Einstein manifolds of positive scalar curvature in dimension four, and give the proof to Theorem \ref{Main5}.

For simplicity, let $(M,g)$ be a four-dimensional manifold with $\Rc=g$. Consequently, $\SS=4$. For such a manifold, $\EE\equiv 0$, so the block matrix for $\hat{\RR}$ in  (\ref{R4block}) is diagonal. Using the notation from Proposition \ref{DerProp} and Theorem \ref{Main4}, the eigenvalues of $\hat{\RR}$ are given by 
\[(\frac{1}{3}-\lambda_i-\mu_j).\]

\begin{proof}[Proof of Theorem \ref{Main5}]
	First, with the aid of the ordering of eigenvalues of $\WW$ in (\ref{eigenvalue}), we have
	\begin{align*}
		\lambda_3+\mu_3 &\geq \lambda_3+\mu_2\geq \lambda_3+\mu_1,\\ \lambda_2+\mu_2 &\geq \lambda_2+\mu_1\geq \lambda_1+\mu_1,\\
		\lambda_3+\mu_3 &\geq \lambda_2+\mu_3\geq \lambda_1+\mu_3,\\ \lambda_2+\mu_2 &\geq \lambda_1+\mu_2\geq \lambda_1+\mu_1. 
	\end{align*}
	$\hat{\RR}$ is 5-non-negative if and only if
	\begin{align*}
		0 &\leq \frac{5}{3}-3\lambda_3-3\mu_3-\lambda_2-\lambda_1-\mu_2-\mu_1,\\
		0 &\leq \frac{5}{3}-3\lambda_3-2\mu_3-2\lambda_2-2\mu_2-\mu_1,\\
		0 &\leq \frac{5}{3}-2\lambda_3-3\mu_3-2\lambda_2-2\mu_2-\lambda_1.
	\end{align*}
	Using $\sum_i \lambda_i=\sum_i \mu_i=0$ and Proposition \ref{berger} yields
	\begin{align*}
		0 &\leq \frac{5}{3}-2(\lambda_3+\mu_3)\\
		0&\leq \frac{5}{3}-4a_3\\
		0&\leq \frac{5}{3}-4\WW_{1414}\\
		0&\leq \frac{5}{3}-4(\RR_{1414}-\frac{1}{3})\\
		0&\leq \frac{5}{3}-4(\RR_{1414}-\frac{1}{3})\\
		\RR_{1414} &\leq \frac{3}{4}.
	\end{align*}
	By the ordering (\ref{eigenvalue}), the sectional curvature is bounded above by $\frac{3}{4}$. Using the classification result of \cite[Corollary 1.3]{caotran18} yields the conclusion. 
\end{proof}

When $\2R$ is $6$-non-negative, we have the following observation.
\begin{proposition}	Let $(M,g)$ be a simply connected Einstein four-manifold with positive scalar curvature. If $\hat{\RR}$ is 6-positive then its sectional curvature is bounded above by the Einstein constant. Moreover, the curvature operator (of first kind) is 4-non-negative. 
\end{proposition}
\begin{proof}
	Again, we use the normalization $\Rc=g$. $\hat{\RR}$ is 6-non-negative if and only if
	\begin{align*}
		0 &\leq 2-3\lambda_3-3\mu_3-2\lambda_2-\lambda_1-2\mu_2-\mu_1\\
		0 &\leq 2-3\lambda_3-2\mu_3-3\lambda_2-2\mu_2-2\mu_1\\
		0 &\leq 2-2\lambda_3-3\mu_3-2\lambda_2-3\mu_2-2\lambda_1.
	\end{align*} 
	Due to Prop. \ref{berger}, it is equivalent to 
	\begin{align*}
		0 &\leq 2-(\lambda_3+\mu_3)+\lambda_1+\mu_1\\
		&\leq 2-2a_3+2a_1\\
		0 &\leq 2+3\lambda_1\\
		0 &\leq 2+3\mu_1.
	\end{align*} 
	The first inequality is equivalent to 
	\[\RR_{1414}-\RR_{1212}\leq 1.\]
	In combination with the equality 
	\[\RR_{1212}+\RR_{1313}+\RR_{1414}=1,\]
	and the ordering
	\[\RR_{1212}\leq \RR_{1313}\leq \RR_{1414},\]
	we conclude that $\RR_{1414}\leq 1$. 
	
	For the last statement, recall that the eigenvalues of the curvature operator of the first kind are given by 
	\begin{align*}
		\lambda_1+\frac{1}{3} &\leq \lambda_2+\frac{1}{3}\leq \lambda_3+\frac{1}{3};\\
		\mu_1+\frac{1}{3}&\leq \mu_2+\frac{1}{3} \leq \mu_3+\frac{1}{3}.
	\end{align*}
	Thus, $\RR$ is 4-non-negative if and only if 
	\begin{align*}
		0 &\leq \frac{4}{3}-\lambda_3-\mu_3,\\
		0 &\leq \frac{4}{3}+\lambda_1,\\
		0 &\leq \frac{4}{3}+\mu_1. 
	\end{align*}
	
	The first inequality is equivalent to 
	\[\RR_{1414}\leq 1.\]
	The result then follows. 
\end{proof}

\subsection{Examples} \label{Examples}  To illustrate our results, we use Theorem \ref{Main4} to compute the matrix of $\hR$ for some model cases. \medskip

\noindent {\bf 1.}  $(S^4, g_0)$, where $g_0$ is the round metric.  In this case $W = 0$ and $S = 12$ at each point, hence
\begin{align*}
	\hR = 4 \mathbb{I},
\end{align*}
where $\mathbb{I}$ is the identity matrix.  In particular, $\hR$ (as a bilinear form) is positive definite.  \medskip

\noindent {\bf 2.} $(\mathbb{CP}^2, g_{FS})$, where $g_{FS}$ is the Fubini-Study metric.  In this case, $W^{-} \equiv 0$ and $S = 8$. Since the metric is K\"ahler, $W^{+}$ can be diagonalized at each point as
\begin{align} \label{WK}
	W^{+} = \left (\begin{array}{lll}
		\frac{1}{6}S & \ & \ \\
		\ &  \frac{1}{12}S & \  \\
		\ & \ & \frac{1}{12}S
	\end{array}
	\right),
\end{align}
see Proposition 2 of \cite{Derd83}.  Consequently, up to ordering of the eigenvalues, the matrix for $\hR$ is given by
\begin{align} \label{CP2matrix}
	\hat{R} = 16 \left (\begin{array}{lll}
		-\frac{1}{2} \mathbb{I} \, & \, 0 \, & \, 0 \, \\
		\ \ \ 0 \, & \, \mathbb{I} \, & \, 0 \, \\
		\ \ \ 0 \, & \, 0 \, & \, \mathbb{I} \,
	\end{array}
	\right).
\end{align}
Note that the sum of the four smallest eigenvalues is negative, but the sum of the five smallest is positive.  Hence $\hR$ is $5$-positive but not $4$-positive.

\medskip

\noindent {\bf 3.}  $(S^2 \times S^2, g_{p})$, where $g_p$ is the product of the standard metric on each factor.  In this case, $S = 4$, and $g_p$ is K\"ahler with respect to both orientations; i.e., the representation (\ref{WK}) holds for both $W^{+}$ and $W^{-}$.  Consequently, up to ordering of the eigenvalues, the matrix for $\hR$ is given by
\begin{align} \label{S2pmatrix}
	\hat{R} = 4 \left (\begin{array}{lllllllll}
		-1 & \ & \ & \ & \ & \ & \ & \ & \ \\
		\ & 0 & \ & \ & \ & \ & \ & \ & \ \\
		\ & \ & 0 & \ & \ & \ & \ & \ & \ \\
		\ & \ & \ & 0 & \ & \ & \ & \ & \ \\
		\ & \ & \ & \ & 0 & \ & \ & \ & \ \\
		\ & \ & \ & \ & \ & 1 & \ & \ & \ \\
		\ & \ & \ & \ & \ & \ & 1 & \ & \ \\
		\ & \ & \ & \ & \ & \ & \ & 1 & \ \\
		\ & \ & \ & \ & \ & \ & \ & \ & 1 \\
	\end{array}
	\right).
\end{align}
Notice that the sum of the five smallest eigenvalues is negative; i.e., $\hR$ is not five-non-negative.  However, it is six-non-negative.  
\bibliographystyle{plain}
\bibliography{bioMorse}

\def\cprime{$'$}
\begin{thebibliography}{10}

\bibitem{berger61}
Marcel Berger.
\newblock Sur quelques vari\'et\'es d'{E}instein compactes.
\newblock {\em Ann. Mat. Pura Appl. (4)}, 53:89--95, 1961.

\bibitem{bohmwilking}
Christoph B{\"o}hm and Burkhard Wilking.
\newblock Manifolds with positive curvature operators are space forms.
\newblock {\em Ann. of Math. (2)}, 167(3):1079--1097, 2008.

\bibitem{bk78}
Jean-Pierre Bourguignon and Hermann Karcher.
\newblock Curvature operators: pinching estimates and geometric examples.
\newblock {\em Ann. Sci. \'{E}cole Norm. Sup. (4)}, 11(1):71--92, 1978.

\bibitem{Brendle08PIC1}
Simon Brendle.
\newblock A general convergence result for the {R}icci flow in higher
  dimensions.
\newblock {\em Duke Math. J.}, 145(3):585--601, 2008.

\bibitem{BrendleDuke}
Simon Brendle.
\newblock Einstein manifolds with nonnegative isotropic curvature are locally
  symmetric.
\newblock {\em Duke Math. J.}, 151(1):1--21, 2010.

\bibitem{Brendle18PIC}
Simon Brendle.
\newblock Ricci flow with surgery on manifolds with positive isotropic
  curvature.
\newblock {\em Ann. of Math. (2)}, 190(2):465--559, 2019.

\bibitem{bs091}
Simon Brendle and Richard Schoen.
\newblock Manifolds with {$1/4$}-pinched curvature are space forms.
\newblock {\em J. Amer. Math. Soc.}, 22(1):287--307, 2009.

\bibitem{caotran4}
Xiaodong Cao and Hung Tran.
\newblock Einstein four-manifolds of pinched sectional curvature.
\newblock {\em Adv. Math.}, 335:322--342, 2018.

\bibitem{caotran18}
Xiaodong Cao and Hung Tran.
\newblock Four-manifolds of pinched sectional curvature.
\newblock {\em arXiv:math.DG/1809.05158}, 2018.

\bibitem{cao14einstein}
Xiaodong Cao and Peng Wu.
\newblock Einstein four-manifolds of three-nonnegative curvature operator.
\newblock {\em Unpublished}, 2014.

\bibitem{CR14}
Ezio Costa and Ernani Ribeiro, Jr.
\newblock Four-dimensional compact manifolds with nonnegative biorthogonal
  curvature.
\newblock {\em Michigan Math. J.}, 63(4):747--761, 2014.

\bibitem{Derd83}
Andrzej Derdzi{\'n}ski.
\newblock Self-dual {K}\"ahler manifolds and {E}instein manifolds of dimension
  four.
\newblock {\em Compositio Math.}, 49(3):405--433, 1983.

\bibitem{gl99}
Matthew~J. Gursky and Claude Le{B}run.
\newblock On {E}instein manifolds of positive sectional curvature.
\newblock {\em Ann. Global Anal. Geom.}, 17(4):315--328, 1999.

\bibitem{H3}
Richard~S. Hamilton.
\newblock Three-manifolds with positive {R}icci curvature.
\newblock {\em J. Differential Geom.}, 17(2):255--306, 1982.

\bibitem{kashiwada93}
Toyoko Kashiwada.
\newblock On the curvature operator of the second kind.
\newblock {\em Natur. Sci. Rep. Ochanomizu Univ.}, 44(2):69--73, 1993.

\bibitem{MM88}
Mario~J. Micallef and John~Douglas Moore.
\newblock Minimal two-spheres and the topology of manifolds with positive
  curvature on totally isotropic two-planes.
\newblock {\em Ann. of Math. (2)}, 127(1):199--227, 1988.

\bibitem{mw93}
Mario~J. Micallef and McKenzie~Y. Wang.
\newblock Metrics with nonnegative isotropic curvature.
\newblock {\em Duke Math. J.}, 72(3):649--672, 1993.

\bibitem{MRS20}
Josef Mike\v{s}, Vladimir Rovenski, and Sergey~E. Stepanov.
\newblock An example of {L}ichnerowicz-type {L}aplacian.
\newblock {\em Ann. Global Anal. Geom.}, 58(1):19--34, 2020.

\bibitem{nishi86positive}
Seiki Nishikawa.
\newblock On deformation of {R}iemannian metrics and manifolds with positive
  curvature operator.
\newblock In {\em Curvature and topology of {R}iemannian manifolds ({K}atata,
  1985)}, volume 1201 of {\em Lecture Notes in Math.}, pages 202--211.
  Springer, Berlin, 1986.

\bibitem{OT79homology}
Koichi Ogiue and Shun-ichi Tachibana.
\newblock Les vari\'{e}t\'{e}s riemanniennes dont l'op\'{e}rateur de courbure
  restreint est positif sont des sph\`eres d'homologie r\'{e}elle.
\newblock {\em C. R. Acad. Sci. Paris S\'{e}r. A-B}, 289(1):A29--A30, 1979.

\bibitem{st69}
I.~M. Singer and J.~A. Thorpe.
\newblock The curvature of {$4$}-dimensional {E}instein spaces.
\newblock In {\em Global {A}nalysis ({P}apers in {H}onor of {K}. {K}odaira)},
  pages 355--365. Univ. Tokyo Press, Tokyo, 1969.

\bibitem{yangdg00}
DaGang Yang.
\newblock Rigidity of {E}instein {$4$}-manifolds with positive curvature.
\newblock {\em Invent. Math.}, 142(2):435--450, 2000.

\end{thebibliography}

\end{document}